\documentclass[journal]{IEEEtran}
\ifCLASSINFOpdf
\usepackage[pdftex]{graphicx}
\else
\usepackage[dvips]{graphicx}
\DeclareGraphicsExtensions{.eps}
\fi
\usepackage{epstopdf}
\usepackage{color}
\usepackage{balance}
\usepackage[cmex10]{amsmath}
\usepackage{amssymb}
\usepackage{amsthm}
\usepackage{mathtools}
\usepackage{bm}
\usepackage{tikz}
\usetikzlibrary{shadows}

\usepackage{multirow}

\usepackage{pgfplots}

\newenvironment{customlegend}[1][]{%
	\begingroup
	\csname pgfplots@init@cleared@structures\endcsname
	\pgfplotsset{#1}%
}{%
\csname pgfplots@createlegend\endcsname
\endgroup
}%

\def\addlegendimage{\csname pgfplots@addlegendimage\endcsname}

\pgfkeys{/pgfplots/number in legend/.style={%
		/pgfplots/legend image code/.code={%
			\node at (0.295,-0.0225){#1};
		},%
	},
}

\theoremstyle{remark} \newtheorem{lemma}{Lemma}
\theoremstyle{remark} \newtheorem{thm}{Theorem}
\theoremstyle{remark} 
\theoremstyle{remark} \newtheorem{prop}{Proposition}
\theoremstyle{remark} \newtheorem{rmk}{Remark}

\DeclareMathOperator{\diag}{diag}

\DeclareMathOperator{\re}{Re}
\DeclareMathOperator{\im}{Im}

\begin{document}
	\allowdisplaybreaks
	\title{A Necessary Condition for Power Flow Insolvability in Power Distribution Systems with Distributed Generators}
	
	\author{Zhaoyu Wang, \IEEEmembership{Member, IEEE}, Bai Cui, \IEEEmembership{Student Member, IEEE}, and Jianhui Wang, \IEEEmembership{Senior Member, IEEE}% <-this % stops a space
		\thanks{The work of J. Wang is supported by the U.S. Department of Energy (DOE)'s Office of Electricity Delivery and Energy Reliability.}% <-this % stops a space
		\thanks{Z. Wang is with the Department of Electrical and Computer Engineering, Iowa State University, Ames, IA 50011 USA (e-mail: wzy@iastate.edu).}% <-this % stops a space
		\thanks{B. Cui is with the School of Electrical and Computer Engineering, Georgia Institute of Technology, Atlanta, GA 30332 USA (e-mail: bcui7@gatech.edu).}% <-this % stops a space
		\thanks{J. Wang is with Argonne National Laboratory, Argonne, IL 60439 USA (e-mail: jianhui.wang@anl.gov).}% <-this % stops a space
	}

	\maketitle
	
	\begin{abstract}
		This paper proposes a necessary condition for power flow insolvability in power distribution systems with distributed generators (DGs). We show that the proposed necessary condition indicates the impending singularity of the Jacobian matrix and the onset of voltage instability. We consider different operation modes of DG inverters, e.g., constant-power and constant-current operations, in the proposed method. A new index based on the presented necessary condition is developed to indicate the distance between the current operating point and the power flow solvability boundary. Compared to existing methods, the operating condition-dependent critical loading factor provided by the proposed condition is less conservative and is closer to the actual power flow solution space boundary. The proposed method only requires the present snapshots of voltage phasors to monitor the power flow insolvability and voltage stability. Hence, it is computationally efficient and suitable to be applied to a power distribution system with volatile DG outputs. The accuracy of the proposed necessary condition and the index is validated by simulations on a distribution test system with different DG penetration levels.
	\end{abstract}
	

	% \fi
	\IEEEpeerreviewmaketitle

	\section{Introduction}
	\IEEEPARstart{T}{HE} increasing penetration of distributed generators (DGs) and the appearance of power electronic loads has imposed new challenges to the modeling, operation, and control of power distribution systems. The traditional analysis and operation paradigm of distribution systems needs to be changed to accommodate these new types of generators and loads. The ability to assess and maintain the security margins within the operational context of the growing deployment of DGs is important to modernized power distribution systems. The solvability of power flow equations is a desirable metric to indicate the security margins of power systems \cite{Overbye1}. This paper provides a necessary condition for power flow insolvability, i.e., a sufficient condition for power flow solvability, which can be used for the fast online assessment of static voltage stability of a power distribution system with a high penetration of inverter-interfaced DGs. 
	
	Because of the nonlinear nature of power flow equations, they are typically solved by Newton-type numerical techniques.
	%The power flow equations are insolvable if no valid voltage profile can be found given a set of power injections. 
	Many conditions have been proposed to guarantee the existence of power flow solutions. The study in \cite{Overbye1} develops a Newton-Raphson-based algorithm to quantify the degree of insolvability by calculating the distance between the desired operating point and the closest solvable operating point. Reference \cite{Ilic1} investigates the conditions under which the unique and operationally acceptable solutions exist for the decoupled active and reactive power flow model. Fixed-point theorems are applied to derive sufficient conditions for the existence of unique power flow solution in \cite{Bolognani2} and \cite{Lesieutre1}. In \cite{Molzahn1}, the power flow solvability problem is formulated as a nonconvex optimization problem. In order to derive a sufficient condition for system insolvability the original problem needs to be convexified so that efficient algorithms can be applied to find the global optimum solution. Semidefinite relaxation technique is thus applied to convert the problem into a convex one, and a sufficient condition for system insolvability is derived. Reference \cite{Chiang1} proves the existence and uniqueness of power flow solutions in radial distribution networks through iterative methods.
	% Reference [4] presents a method to restore power flow solvability and mitigate unsolvable cases by quantifying the sensitivities of insolvability to different system controls. The study in [5] proposes a direct interior point method to restore the system solvability by controlling generator outputs, tap positions of load tap changers, and load shedding.
	%Reference [6] proposes a predictor-corrector continuation method to explore the power flow solution space boundary. 
	Reference \cite{Grijalva1} presents a necessary condition for power flow insolvability and demonstrates that at least one branch must reach its static transfer stability limit before the singularity of Jacobian matrix is reached. A recent work \cite{Simpson1} proposes a multi-bus short-circuit ratio to quantify the stress the grid is under. Qualitatively it is similar to the condition in \cite{Bolognani2}, but instead of separating the network from the loading, it combines them in a matrix-vector product. The integration of DGs further complicates the problems of power flow solvability and voltage stability. Studies have shown that voltage stability issues do exist in distribution networks \cite{Kersting1, Yan1, Liu1}. 
	
	%The exact determination of the margin to voltage instability of a power system requires continuation power flow solutions for multiple system scenarios [11]. The distance to the solvability boundary has been used as an indicator of proximity to voltage stability margins of power systems. When a power system reaches its point of static voltage collapse, the power flow becomes insolvable due to the singularity/saddle-node bifurcation of the Jacobian matrix or the control limits of the system (limit-induced bifurcation) [11]. Several papers have investigated the power flow insolvability with applications to voltage stability margins. Reference [12] applies invariant subspace parametric sensitivity to perform load shedding and generator re-dispatch to restore the solvability of a power system described by differential-algebraic equations. The study in [9] presents a sufficient condition under which the nonlinear power flow problem has no solution, and calculates voltage stability margins based on the distance to the power flow solution boundary. Reference [13] proposes an iterative methods to compute the load power at which bifurcation occurs, and calculates the distance in load power parameter space to the bifurcation point as an index of voltage collapse.
	
	This paper investigates the power flow solvability in a power distribution system with DGs. We propose a necessary condition for power flow insolvability due to saddle-node bifurcation. A saddle-node bifurcation occurs when two system equilibrium points coalesce and annihilate each other under slow parameter changes \cite{Kwatny1, Canizares1}. The saddle-node bifurcation phenomenon that we are interested in is concerned with the disappearance of normal power system operating points as system stresses under gradual load increase. As demonstrated in previous results \cite{Dobson1} the bifurcation points are irrelevant to load dynamics and correspond to points where solutions of the algebraic power flow equations are lost. A detailed theoretical proof shows that the proposed condition can be used to analyze a power distribution system with DGs that operate in different modes, e.g., constant-power and constant-current modes \cite{Wang1, Morsy1}, where we assume that for these operation modes, the real and imaginary parts of power/current are given. Based on the necessary condition, we design a new index to monitor the operating condition of a distribution system with a high DG penetration level, which can provide an accurate precursor to system overloading, when the generators are modeled as constant-current and/or constant-power sources. In comparison with existing methods, the proposed method can provide an accurate assessment which is less conservative, closer to the actual power flow solvability boundary and adaptive to the present operating point. The calculation of the proposed necessary condition and index only requires a snapshot of the present bus voltage and current. In addition, the proposed method requires a small computation effort, which makes it ideal for real-time applications.
	
	The remainder of this paper is organized as follows. Section \ref{modelling} introduces the power distribution system model and the proposed necessary condition. Section \ref{main} provides the theoretical proof for our proposed condition on power distribution systems with DGs. In Section \ref{simulations}, the numerical results are provided. Section \ref{discussion} discusses the simulation results and the physical implication of the proposed index. Section \ref{conclsn} concludes the paper with major findings.
	
	\section{Distribution System Model and Proposed Necessary Condition}
	\label{modelling}
	We conduct a per-phase analysis to a power distribution system with $n+1$ buses. The line section between buses $i$ and $k$ in the system is weighted by its complex admittance $y_{ik} \coloneqq 1/z_{ik} = g_{ik} + jb_{ik}$.
	
	It is assumed that the system contains a single substation which is modeled as a voltage-regulated source, i.e., the slack bus \cite{Kersting1}. The phase angle of the slack bus is fixed as a reference. Without loss of generality, we assume the slack bus has a voltage phasor $V_{\mathcal{S}} = 1\angle 0^{\circ}$. In addition, the system has $g$ DGs and $l = n - g$ loads. Tie buses that neither inject nor absorb power are assumed to be eliminated via standard methods such as that in \cite{Bolognani1}. Let the set of DG buses be $\mathcal{G} = \{ 1, \ldots, g \}$ and the set of load buses be $\mathcal{L} = \{ g+1, \ldots, n \}$. For PQ buses, the injected power is given by $S_i = P_i + jQ_i, \forall i \in \{\mathcal{G}, \ \mathcal{L}\}$.
	
	The system can be represented by the following equation
	\begin{equation}
		\begin{bmatrix}
			I_{\mathcal{S}} \\
			I
		\end{bmatrix} = 
		\begin{bmatrix}
			Y_{\mathcal{SS}} & Y_{\mathcal{SL}} \\
			Y_{\mathcal{LS}} & Y_{\mathcal{LL}}
		\end{bmatrix}
		\begin{bmatrix}
			V_{\mathcal{S}} \\
			V
		\end{bmatrix}
		\label{admit}
	\end{equation}
	where $I_{\mathcal{S}} \in \mathbb{C}^1$ is the slack bus current, $I = I_D + jI_Q \in \mathbb{C}^n$ is the vector of generator and load currents, and $V \in \mathbb{C}^n$ is the vector of generator and load voltages. The polarity of the currents is assumed to be out of the network through the buses.
	We obtain from (\ref{admit}) that
	\begin{equation}
		V = -Y_{\mathcal{LL}}^{-1}Y_{\mathcal{LS}}V_{\mathcal{S}} + Y_{\mathcal{LL}}^{-1}I
		\label{equi1}
	\end{equation}
	Define the vector of equivalent voltage to be $E \coloneqq -Y_{\mathcal{LL}}^{-1} Y_{\mathcal{LS}} V_{\mathcal{S}}$ and the impedance matrix to be $Z \coloneqq -Y_{\mathcal{LL}}^{-1}$. With the definitions, ({\ref{equi1}}) can be rewritten as
	\begin{equation}
		V = E - ZI
		\label{multiport}
	\end{equation}
	
	Given the bus power injection vector $S$, the vectors of voltage and current are related by 
	\begin{equation}
		S = \diag (I^*)V = \diag (I^*)E - \diag (I^*)ZI
		\label{ohm}
	\end{equation}
	where $I^*$ is the vector of complex conjugate of $I$ and $\diag (\cdot)$ denotes the diagonal matrix whose diagonal elements are the entries of the vector. The elements in (\ref{ohm}) can be written as
	\begin{subequations}
		\begin{align}
			\label{real}
			P_h & = \re\left( I_h^*E_h - I_h^*\sum\nolimits_{i=1}^n Z_{hi}I_i \right), \\
			\label{reac}
			Q_h & = \im\left( I_h^*E_h - I_h^*\sum\nolimits_{i=1}^n Z_{hi}I_i \right),
		\end{align}
	\end{subequations}
	where $h \in \{\mathcal{G},\mathcal{L}\}$ is either load bus or constant-power DG bus. Equations (\ref{real})--(\ref{reac}) define the power flow equations of the distribution system parameterized by bus current injections in rectangular coordinates. The adoption of current injections as state vectors can facilitate our derivation of the necessary condition.
	
	The power flow Jacobian of the model in (\ref{real})--(\ref{reac}) is defined as
	\begin{equation}
		J^{\mathcal{R}} = 
		\begin{bmatrix}
			\partial P / \partial I_D & \partial P/ \partial I_Q \\
			\partial Q / \partial I_D & \partial Q / \partial I_Q
		\end{bmatrix}.
		\label{jacob}
	\end{equation}
	The singularity of the conventional power flow Jacobian matrix---The Jacobian matrix whose power flow equations are parameterized by bus voltage phasors in polar coordinates---is commonly used as a necessary condition to indicate system loadability limit, which in turn marks the onset of voltage instability for a system with PQ buses \cite[Ch.~7]{Cutsem1}. Singularity of (\ref{jacob}) coincides with that of the conventional Jacobian matrix $J_{\mathrm{conv}}$ due to the chain rule,
	\begin{equation}
	J_{\mathrm{conv}} =J^{\mathcal{R}}
	\begin{bmatrix}
	\partial I_D / \partial |V| & \partial I_D / \partial \theta \\
	\partial I_Q / \partial |V| & \partial I_Q / \partial \theta
	\end{bmatrix}
	\end{equation}
	where $|V|$ and $\theta$ are vectors that represent element-wise magnitudes and angles of the voltage vector $V$, i.e., $|V|_i\coloneqq |V_i|$, $\theta_i \coloneqq \angle V_i$. Thus the singularity of (\ref{jacob}) can be used as an indicator of voltage instability.
	
    We propose the following necessary condition for the singularity of (\ref{jacob}),
	\begin{equation}
	\label{cond}
	\exists h \in \{\mathcal{L}, \mathcal{G} \} \text{ such that } |V_h| \leq \sum_{i = 1}^n |Z_{hi}I_i|, 
	\end{equation}
	where $h$ is either load bus or constant-power DG bus. The necessary condition (\ref{cond}) relates to the singularity of the Jacobian matrix (\ref{jacob}). The next step is to show that (\ref{jacob}) is always non-singular unless condition (\ref{cond}) is satisfied.
	
	Based on the necessary condition (\ref{cond}), an index that measures the criticality of system loading condition is proposed. The index is called $C$-index, and is defined as
	\begin{equation}
	C_h = \frac{|V_h|}{\sum_{i = 1}^n |Z_{hi}I_i|}, \ h \in \{\mathcal{L}, \mathcal{G} \}
	\end{equation}
	The system-wide $C$-index is defined as $C = \min\{C_h\}$. The system loadability limit is reached only if $C < 1$.

	\section{Main Result in Distribution System with DGs}
	\label{main}
	In classic power flow analysis, voltages are usually represented in polar coordinates, power is represented in rectangular coordinates, and Jacobian matrices are represented as real-valued matrices. As our subsequent formulation shows, the adoption of current injections as state variables in the power flow formulation relates the entries of power flow Jacobian with voltages, which can assist our analysis. We would like to demonstrate our approach on power power distribution networks with DGs by expressing power injections by currents as in (\ref{ohm}) and forming the power flow Jacobian matrix by taking partial derivatives as in (\ref{jacob}). However, the problem with this approach is that the entries in the Jacobian matrix do not have direct physical interpretations in an AC network, which makes it difficult to draw connections between (\ref{cond}) and the singularity of Jacobian matrix. 
	
	To solve the above mentioned challenge, we propose to formulate power flow Jacobian as a complex matrix via Wirtinger Calculus \cite{Remmert1, Hjorungnes1}. 
	
	\subsection{Wirtinger Calculus}
	\label{wirtinger}
	Given a complex function $f = u(x, y) + jv(x, y): \mathbb{C} \to \mathbb{C}$, $f$ is complex-differentiable ($\mathbb{C}$-differentiable) if it satisfies the Cauchy-Riemann condition, i.e.,
	\begin{equation}
	\frac{\partial u}{\partial x} = \frac{\partial v}{\partial y}, \quad \frac{\partial u}{\partial y} = -\frac{\partial v}{\partial x}
	\end{equation}
	
	Based on the condition, it can be verified that power flow equations are generally not $\mathbb{C}$-differentiable. Specifically, complex conjugation $g(z) = z^* = x - jy$ is not $\mathbb{C}$-differentiable as
	\begin{equation}
	1 = \frac{\partial u}{\partial x} \neq \frac{\partial v}{\partial y} = -1
	\end{equation}
	
	Given a complex function $f: \mathbb{C} \to \mathbb{C}$, we define the function $F: \mathbb{R}^2 \to \mathbb{C}$ as $F(x,y) = U(x,y) + jV(x,y) = f(x+jy)$. The function $f$ is said to be $\mathbb{R}$-differentiable if
	\begin{equation}
	\frac{\partial U}{\partial x}, \quad \frac{\partial U}{\partial y}, \quad \frac{\partial V}{\partial x}, \quad \frac{\partial V}{\partial y}
	\end{equation}
	all exist.
	
	Assume $f$ is $\mathbb{R}$-differentiable, the total derivative of $F$ is given by
	\begin{align}
	\mathrm{d}F &= \left( \frac{\partial U(x,y)}{\partial x} + j\frac{\partial V(x,y)}{\partial x} \right)\mathrm{d}x \nonumber\\
		& \qquad + \left( \frac{\partial U(x,y)}{\partial y} + j\frac{\partial V(x,y)}{\partial y} \right)\mathrm{d}y \nonumber\\
	& = \frac{\partial F(x,y)}{\partial x}\mathrm{d}x + \frac{\partial F(x,y)}{\partial y}\mathrm{d}y 
	\label{totalderiv}
	\end{align}
	We define
	\begin{subequations}
	\begin{align}
	\mathrm{d}z & = \mathrm{d}x + j\mathrm{d}y \\
	\mathrm{d}z^* & = \mathrm{d}x - j\mathrm{d}y
	\end{align}
	\label{dz}
	\end{subequations}
	Then the two differentials $\mathrm{d}x$ and $\mathrm{d}y$ are solved for as
	\begin{subequations}
		\begin{align}
		\mathrm{d}x & = \frac{1}{2}\left( \mathrm{d}z + \mathrm{d}z^* \right) \\
		\mathrm{d}y & = \frac{j}{2}\left( \mathrm{d}z^* - \mathrm{d}z \right)
		\end{align}
		\label{dxy}
	\end{subequations}
	Substituting (\ref{dxy}) into (\ref{totalderiv}) and rearranging terms gives
	\begin{align}
	\mathrm{d}F &= \frac{1}{2}\frac{\partial F}{\partial x}(\mathrm{d}z + \mathrm{d}z^*) + \frac{j}{2}\frac{\partial F}{\partial y}(\mathrm{d}z^* - \mathrm{d}z) \nonumber\\
	&= \frac{1}{2} \left( \frac{\partial F}{\partial x} - j\frac{\partial F}{\partial y} \right) \mathrm{d}z + \frac{1}{2}\left( \frac{\partial F}{\partial x} + j\frac{\partial F}{\partial y} \right) \mathrm{d}z^*
	\end{align}
	Motivated by the above formulation, we introduce the `complex partial differential' operators as
	\begin{subequations}
		\begin{align}
		\frac{\partial}{\partial z} & = \frac{1}{2} \left( \frac{\partial}{\partial x} - j\frac{\partial}{\partial y} \right) \\
		\frac{\partial}{\partial z^*} & = \frac{1}{2} \left( \frac{\partial}{\partial x} + j\frac{\partial}{\partial y} \right)
		\end{align}
	\end{subequations}
	Based on the definition, it is easy to verify that the differential operators of the conjugate function $f^*(z)$ satisfy
	\begin{subequations}
		\begin{align}
		\frac{\partial f^*(z)}{\partial z} & = \left( \frac{\partial f(z)}{\partial z^*} \right)^* \\
		\frac{\partial f^*(z)}{\partial z^*} & = \left( \frac{\partial f(z)}{\partial z} \right)^*
		\end{align}
	\end{subequations}
	With the above definitions, the differential $\mathrm{d}f$ can be defined as
	\begin{equation}
	\mathrm{d}f = \frac{\partial f(z)}{\partial z} \mathrm{d}z + \frac{\partial f(z)}{\partial z^*} \mathrm{d}z^*
	\label{diff}
	\end{equation}
	
	\begin{rmk}
	Notice that (\ref{diff}) is defined formally. However, from a geometrical point of view \cite{Lee1}, $\mathrm{d}f$ is a complex-valued differential one-form on $\mathbb{C}$. That is, it is an $\mathbb{R}$-linear operator at $z$ from the tangent space $T_{z}\mathbb{C} \cong \mathbb{C}$ to $\mathbb{C}$. With $z:\mathbb{C} \to \mathbb{C}$ and $z^*: \mathbb{C} \to \mathbb{C}$ as identity and complex conjugate functions respectively, $\mathrm{d}z$ and $\mathrm{d}z^*$ are also one-forms and they form a basis for the complexified cotangent space at every point $z$. The operators $\frac{\partial}{\partial z}$ and $\frac{\partial}{\partial z^*}$ are vectors on the complexified tangent space at every point $z$ and they form a basis which is dual to the basis $\{ \mathrm{d}z, \mathrm{d}z^* \}$. For instance,
	\begin{align}
	\mathrm{d}z\left(\frac{\partial}{\partial z}\right) &= \frac{1}{2}(\mathrm{d}x + j\mathrm{d}y) \left( \frac{\partial}{\partial x} - j\frac{\partial}{\partial y} \right) \nonumber\\
	& = \frac{1}{2}\left( \mathrm{d}x\frac{\partial}{\partial x} + \mathrm{d}y\frac{\partial}{\partial y} \right) \nonumber\\
	& = 1
	\end{align}
	
	The various operators can also be derived by noting the isomorphism between the real vector space $\mathbb{R}^2$ and the complex vector space $\mathbb{C}$ \cite{Huybrechts1}.
	\end{rmk}
	
	\subsection{Application to Power Flow Analysis}
	
	Given a power distribution system with $n+1$ buses, in which there are $n$ $PQ$ buses and one slack bus. we may define the complex power flow Jacobian as
	\begin{equation}
	J^{\mathcal{Z}} = 
	\begin{bmatrix}
	\partial S / \partial I & \partial S / \partial I^* \\
	\partial S^* / \partial I & \partial S^* / \partial I^*
	\end{bmatrix}
	\label{comp_jacob}
	\end{equation}
	Notice that the matrix is complex and the dimension of the matrix is $2n \times 2n$, i.e., $J^{\mathcal{Z}} \in \mathbb{C}^{2n\times 2n}$. We have the following equation based on definitions of the differentials
	\begin{equation}
	\begin{bmatrix}
	\mathrm{d}S \\
	\mathrm{d}S^*
	\end{bmatrix} =
	J^{\mathcal{Z}}
	\begin{bmatrix}
	\mathrm{d}I \\
	\mathrm{d}I^*
	\end{bmatrix}
	\label{complexpf}
	\end{equation}
	
	The next theorem shows that the determinant of the new Jacobian matrix defined in (\ref{comp_jacob}) and the original one given in (\ref{jacob}) are identical since they are the same linear operator under two different bases.
	
	\begin{thm}
		Given a power system with $n$ PQ buses and one slack bus, the determinant of the complex power flow Jacobian (\ref{comp_jacob}) and that of the Jacobian (\ref{jacob}) are identical, i.e., $\det J^{\mathcal{Z}} = \det J^{\mathcal{R}}$.
	\label{equaljacob}
	\end{thm}
	
	\begin{proof}
		Define the matrix $T$ as
		\begin{equation}
		T = \begin{bmatrix}
		\frac{1}{2}I & \frac{1}{2}I \\
		-\frac{j}{2}I & \frac{j}{2}I
		\end{bmatrix}
		\end{equation}
		where $I$ is an $n\times n$ identity matrix.
		
		It is known from section \ref{wirtinger} that for a PQ bus $i$,
		\begin{subequations}
			\begin{align}
				\mathrm{d}P_i & = \frac{1}{2}\left( \mathrm{d}S_i + \mathrm{d}S^*_i \right) \\
				\mathrm{d}Q_i & = -\frac{j}{2}\left( \mathrm{d}S_i - \mathrm{d}S^*_i \right)
			\end{align}
		\end{subequations}
		from which we have
		\begin{equation}
		\begin{bmatrix}
		\mathrm{d}P \\
		\mathrm{d}Q
		\end{bmatrix} =
		T
		\begin{bmatrix}
		\mathrm{d}S \\
		\mathrm{d}S^*
		\end{bmatrix}
		\label{transform1}
		\end{equation}
		Similarly,
		\begin{equation}
		\begin{bmatrix}
		\mathrm{d}I_D \\
		\mathrm{d}I_Q
		\end{bmatrix} =
		T
		\begin{bmatrix}
		\mathrm{d}I \\
		\mathrm{d}I^*
		\end{bmatrix}
		\label{transform2}
		\end{equation}
		Substituting (\ref{transform1})--(\ref{transform2}) into (\ref{complexpf}) and rearranging terms gives
		\begin{equation}
		\begin{bmatrix}
		\mathrm{d}P \\
		\mathrm{d}Q
		\end{bmatrix} =
		T J^{\mathcal{Z}} T^{-1}
		\begin{bmatrix}
		\mathrm{d}I_D \\
		\mathrm{d}I_Q
		\end{bmatrix}
		\end{equation}
		We notice that the matrix $TJ^{\mathcal{Z}}T^{-1}$ is simply the Jacobian matrix defined in (\ref{jacob}). That is, the two matrices are similar as
		\begin{equation}
		J^{\mathcal{R}} = TJ^{\mathcal{Z}}T^{-1}
		\end{equation}
		Based on the result from matrix analysis we have
		\begin{equation}
		\det J^{\mathcal{R}} = \det (TJ^{\mathcal{Z}}T^{-1}) = \det J^{\mathcal{Z}}.
		\end{equation}
	\end{proof}
	
	With Theorem \ref{equaljacob}, the voltage stability of a power system can now be examined by checking the singularity of the complex Jacobian matrix (\ref{comp_jacob}). To explore its properties, we write the submatrices of (\ref{comp_jacob}) explicitly as
	\begin{subequations}
		\begin{align}
			\frac{\partial S}{\partial I} & = 
			\begin{bmatrix}
				-Z_{11}I_1^* & -Z_{12}I_1^* & \cdots & -Z_{1n}I_1^* \\
				-Z_{21}I_2^* & -Z_{22}I_2^* & \cdots & \vdots \\
				\vdots 		 & \vdots 		& \ddots & \vdots \\
				-Z_{n1}I_n^* & \cdots 		& \cdots & -Z_{nn}I_n^*
			\end{bmatrix} \label{dSdI} \\	
			\frac{\partial S}{\partial I^*} & = 
			\begin{bmatrix}
			E_1 - \sum\limits_{j=1}^nZ_{1j}I_j 	& 0 		& \cdots & 0 \\
			\vdots 									& \ddots 	&  		 & \vdots \\
			\vdots 		 							&  			& \ddots & \vdots \\
			0 			& \cdots 	& \cdots & E_n - \sum\limits_{j=1}^nZ_{nj}I_j
			\end{bmatrix} \label{dSdIstar}
		\end{align}
	\end{subequations}
	and the submatrices $\partial S^* / \partial I$ and $\partial S^* / \partial I^*$ are element-wise complex conjugates of $\partial S / \partial I^*$ and $\partial S / \partial I$, respectively.
	
	It is noted that both $\partial S / \partial I^*$ and $\partial S^* / \partial I$ are diagonal matrices, and the diagonal element of the $i$th row of $\partial S / \partial I^*$ is the voltage phasor at bus $i$. To prove that (\ref{cond}) is indeed the necessary condition for voltage instability, we define a new matrix $J^{\mathcal{Z'}}$, whose diagonal elements are bus voltage phasors and the sum of the off-diagonal elements are equivalent voltage drops between equivalent voltage sources $K_iE$ and the bus voltages. It is necessary to show that the determinant of the new matrix is related to $J^{\mathcal{Z}}$. Then complex Levy--Desplanques theorem can be applied to prove the necessary condition (\ref{cond}).
	
	Note that interchanging the left block and right block of $J^{\mathcal{Z}}$ changes the sign of the determinant only when $n$ is odd since interchanging two columns of a matrix changes the sign of its determinant. Let the matrix after the block swapping be $J^{\mathcal{Z''}}$,
	\begin{equation}
	J^{\mathcal{Z''}} =  \begin{bmatrix}
	\partial S / \partial I^* & \partial S / \partial I \\
	\partial S^* / \partial I^* & \partial S^* / \partial I
	\end{bmatrix}
	\end{equation}
	and we have
	\begin{equation}
	\det \left(J^{\mathcal{Z''}}\right) = (-1)^n \det \left( \begin{bmatrix}
	\partial S / \partial I & \partial S / \partial I^* \\
	\partial S^* / \partial I & \partial S^* / \partial I^*
	\end{bmatrix} \right)
	\end{equation}
	
	Define the matrix $J^{\mathcal{Z'}}$ by replacing $\partial S / \partial I^*$ and $\partial S^* / \partial I$ by $B$ and $C$ of the same size as
	\begin{equation}
	J^{\mathcal{Z'}} =  \begin{bmatrix}
	\partial S / \partial I^* & B \\
	C 						  & \partial S^* / \partial I
	\end{bmatrix}
	\label{Jzp}
	\end{equation}
	where the matrices $B$ and $C$ are
	\begin{subequations}
		\begin{align}
			B & = 
			\begin{bmatrix}
			-Z_{11}I_1 & -Z_{12}I_2 	& \cdots & -Z_{1n}I_n \\
			-Z_{21}I_1 & -Z_{22}I_2 	& \cdots & \vdots \\
			\vdots 	   & \vdots 		& \ddots & \vdots \\
			-Z_{n1}I_1 & \cdots 		& \cdots & -Z_{nn}I_n
			\end{bmatrix} \label{mat_b} \\	
			C & = B^*
		\end{align}
	\end{subequations}
	
	The next lemma shows that the determinant of $J^{\mathcal{Z'}}$ is equal to that of $J^{\mathcal{Z''}}$, whose absolute value is equal to the determinant of $J^{\mathcal{Z}}$.
	
	\begin{lemma}
		$\det J^{\mathcal{Z}} = (-1)^n \det J^{\mathcal{Z'}} $.
		\label{lemma2}
	\end{lemma}
	\begin{proof}
		Define a $2n \times 2n$ complex-valued matrix $M$ as
		\begin{equation}
		M = 
		\begin{bmatrix}
		M_{11} & M_{12} \\
		M_{21} & M_{22}
		\end{bmatrix}
		\end{equation}
		where the four $n \times n$ blocks are
		\begin{subequations}
			\begin{align}
			M_{11} & =
			\begin{bmatrix}
			V_1 / I_1^*	& 0 			& \cdots 	& 0 \\
			\vdots 		& V_2 / I_2^* 	&  		 	& \vdots \\
			\vdots 		&  				& \ddots 	& \vdots \\
			0 			& \cdots 		& \cdots 	& V_n / I_n^*
			\end{bmatrix} \\
			M_{12} & = 
			\begin{bmatrix}
			-Z_{11} 	& -Z_{12} 		& \cdots 	& -Z_{1n} \\
			-Z_{21} 	& -Z_{22} 		& \cdots 	& \vdots \\
			\vdots 		& \vdots 		& \ddots 	& \vdots \\
			-Z_{n1} 	& \cdots 		& \cdots 	& -Z_{nn}
			\end{bmatrix} \\
			M_{21} & = M_{12}^* \\
			M_{22} & = M_{11}^*
			\end{align}
		\end{subequations}
		where $V_i = K_iE - \sum\limits_{j=1}^nZ_{ij}I_j$ is the bus $i$ voltage.
		
		In addition, define the $2n \times 2n$ complex-valued diagonal matrix $N$ as
		\begin{equation}
		N = 
		\begin{bmatrix}
		I_1^*	& 0			&  			& \cdots	& 			& 0 	 \\
		0		& \ddots	&  			& 			& 			& 0		 \\
				& 			& I_n^*		& 			& 0			& 		 \\
		\vdots	& 			& 			& I_1		& 			& \vdots \\
				& 			& 0			& 			& \ddots	& 		 \\
		0		& 0			& \cdots 	& 			& 			& I_n
		\end{bmatrix}
		\end{equation}
		
		Then we can see that
		\begin{subequations}
			\begin{align}
				J^{\mathcal{Z'}}  & = MN \\
				J^{\mathcal{Z''}} & = NM
			\end{align}
		\end{subequations}
		Therefore,
		\begin{equation}
		\det J^{\mathcal{Z'}} = \det(M)\det(N) = \det J^{\mathcal{Z''}}
		\end{equation}
		Since 
		\begin{equation}
		\det J^{\mathcal{Z}} = (-1)^n \det J^{\mathcal{Z''}}
		\end{equation}
		We arrive at the conclusion that 
		\begin{equation}
		\det J^{\mathcal{Z}} = (-1)^n \det J^{\mathcal{Z'}}
		\end{equation}
	\end{proof}
	
	Now the necessary condition (\ref{cond}) can be easily seen by applying complex Levy--Desplanques theorem on the matrix $J^{\mathcal{Z'}}$. We state the fact as:
	\begin{thm}
		For the $(n+1)$-bus power system described in Section \ref{modelling}, a power injection is at the power flow solvability boundary only when the power flow solution satisfies
		\begin{equation}
			\exists h \in \{\mathcal{L}, \mathcal{G} \} \text{ such that } |V_h| \leq \sum_{i = 1}^n |Z_{hi}I_i|
		\end{equation}
	\end{thm}
	\begin{proof}
		With Theorem \ref{equaljacob} and Lemma \ref{lemma2}, and the continuity of power flow equations, we only need to show the matrix $J^{\mathcal{Z'}}$ is non-singular when $|V_h| >  \sum_{i = 1}^n |Z_{hi}I_i|, \ h \in \{\mathcal{L}, \mathcal{G} \}$. This is a direct consequence of complex Levy--Desplanques theorem, which states that strictly diagonally dominant matrices are non-singular.
	\end{proof}
	
	\begin{rmk}
		\label{rmk2}
		The proposed condition provides a precursor for power flow insolvability by setting an operating condition-dependent upper bound in $(n+1)$-dimensional power injection space. Some fixed-point theorem-based solution existence conditions tend to be conservative. We will show that the upper bound provided by the proposed condition is always greater than the one given by the condition proposed in \cite{Bolognani2}, where the solvability condition is given by
		\begin{equation}
		|V_\mathcal{S}|^2 > 4\|W^{-1}Z (W^*)^{-1}\|^*\|S\|
		\label{Bolognanicond}
		\end{equation}
		where $W = \diag (-Y_{\mathcal{LL}}^{-1}Y_{\mathcal{LS}})$, $\|\cdot\|$ is the Euclidean norm on $\mathbb{C}^n$ and the matrix norm $\|\cdot\|^*$ on $\mathbb{C}^{n \times n}$ is defined as
		\begin{equation*}
		\|A\|^* \coloneqq \max_h \|A_{h\bullet}\| = \max_h \sqrt{\sum_k |A_{hk}|^2}
		\end{equation*}
		where the notation $A_{h\bullet}$ stands for the $h$th row of $A$.
		
		For all $h \in \{\mathcal{L}, \mathcal{G}\}$, 
		\begin{align}
		\sum_{j=1}^n \left| Z_{hj} \frac{S_j}{V_j} \right|
		& = |W_{hh}| \sum_{j=1}^n \left| W_{hh}^{-1} Z_{hj} (W_{jj}^*)^{-1} \left( \frac{S_j W^*_{jj}}{V_j} \right)\right| \nonumber\\
		& \leq |W_{hh}| \left\|\frac{W}{V}\right\|_{\infty} \sum_{j=1}^n \left| W_{hh}^{-1} Z_{hj} (W_{jj}^*)^{-1}S_j \right| \nonumber\\
		& \leq |W_{hh}| \left\|\frac{W}{V}\right\|_{\infty} \left\| \left(W^{-1}Z (W^*)^{-1}\right)_{h\bullet} \right\| \| S \| \nonumber\\
		& \leq |W_{hh}| \left\|\frac{W}{V}\right\|_{\infty} \left\| W^{-1}Z (W^*)^{-1} \right\|^* \| S \| \label{inequal}
		\end{align}
		where $W/V$ is the $n$-dimensional vector such that $(W/V)_i = W_{ii}/V_i$, the second inequality is due to Cauchy-Schwarz and the third from the definition of $\|\cdot\|^*$.
		
		The inequality (\ref{inequal}) suggests that the proposed condition (\ref{cond}) guarantees solvability when
		\begin{equation}
		\left\| W^{-1}Z (W^*)^{-1} \right\|^* \| S \| < \frac{1}{\left\| W/V \right\|_{\infty}^2},
		\label{cond_relax}
		\end{equation}
		since (\ref{inequal}) and (\ref{cond_relax}) lead to
		\begin{equation}
		\sum_{j=1}^n \left| Z_{hj} \frac{S_j}{V_j} \right| < \frac{|W_{hh}|}{\left\| W/V \right\|_{\infty}} \leq |V_h|.
		\label{suffcond}
		\end{equation}
		
		By comparing (\ref{Bolognanicond}) and (\ref{cond_relax}), it is concluded that the upper bound provided by the proposed condition is greater than (\ref{Bolognanicond}) if $1/\left\| W/V\right\|_{\infty} > |V_{\mathcal{S}}|/2$. We claim that this is the only relevant case since solvability conditions defined in (\ref{Bolognanicond}) and (\ref{cond_relax}) are both violated otherwise. This is made clear by the following proposition:
		\begin{prop}
		Assuming a stable high-voltage solution exists for power injection $S$, then $\left\| W^{-1}Z (W^*)^{-1} \right\|^* \| S \| \geq |V_{\mathcal{S}}|^2/4$ when $|V_{\mathcal{S}}| / 2 \geq 1/\left\| W/V\right\|_{\infty}$.
		\end{prop}
		\begin{proof}
		We may assume there exists a power injection vector $S$ and corresponding voltage profile $V$ such that
		\begin{equation}
		\left\| W^{-1}Z (W^*)^{-1} \right\|^* \| S \| < \frac{|V_{\mathcal{S}}|^2}{4}
		\label{hypo}
		\end{equation}
		and
		\begin{equation}
		\frac{|V_{\mathcal{S}}|}{2} \geq \frac{1}{\left\| W/V \right\|_{\infty}}
		\end{equation}
		
		The system has a zero power injection solution where bus voltages $V_0$ are close to $E$ and $1/\left\| W/V_0\right\|_{\infty} > |V_{\mathcal{S}}|/2$ when shunt elements are not extraordinarily large \cite{Molzahn1}. Then, by continuity, there exists a real number $0 < \mu < 1$ such that the voltage profile $V'$ when power injection is $\mu S$ satisfies
		\begin{equation}
		\frac{|V_{\mathcal{S}}|}{2} = \frac{1}{\left\| W/V' \right\|_{\infty}}
		\end{equation}
		while (\ref{inequal}) requires
		\begin{equation}
		\left\| W^{-1}Z (W^*)^{-1} \right\|^* \| \mu S \| \geq \frac{1}{\left\| W/V' \right\|_{\infty}^2}
		\end{equation}
		which contradicts the assumption (\ref{hypo}).
		\end{proof}
	\end{rmk}
	
	\subsection{Influence of System Parameter Perturbations on $C$-index}
	Now we explore how different system parameters influence the $C$-index through sensitivity analysis on a linearized power flow model. Specifically, we consider the sensitivity of line impedance and load power factors. To simplify the argument, we assume throughout the subsection that the distribution system is composed of a slack bus and $n$ PQ buses with inductive loads.
	
	First, we consider the impact of the homogeneous change in line impedance. Assume that for each line and each shunt capacitance the impedance is changed such that $Y_{ij}'/Y_{ij} =a$ for all entries of the admittance matrix, where $Y_{ij}'$ is the $ij$-entry of the new admittance matrix and $a$ is a real number between 0 and 1. Consequently, the new impedance matrix $Z'$ is $Z_{ij}' =Z_{ij}/a$. This change is equivalent to extending the length of each line by multiple of $1/a$. It is expected that the critical loading factor under the new condition will decrease as line losses increase with the increase of the line length. We note that with the change of line impedance, the voltage profile changes under the same loading condition. To this end we propose to apply a linearized power flow approximation which has been validated for distribution system analysis, to derive an approximate voltage solution with new impedance matrix given the loading condition. The following linearized power flow from \cite{Bolognani2} is used: 
	\begin{equation}
	V_j' = V_{\mathcal{S}}\left( 1 + \frac{1}{|V_{\mathcal{S}}|^2}\sum_{i=1}^n Z_{ji}'S_i^* \right), \qquad \forall j \in \mathcal{L}.
	\end{equation}
	where $V_j'$ is the approximate bus voltage at bus $j$ based on the new impedance matrix $Z'$. We assume that increasing load real and reactive power injections causes decrease in PQ bus voltages, which is a condition commonly used for the characterization of stable systems \cite{Taylor1}. Based on the linearized power flow equation, this can be expressed as 
	\begin{align}
	\left| 1 + \frac{1}{|V_{\mathcal{S}}|^2}\sum_{i=1}^n Z_{ji}'(bS_i^*) \right| < 	\left| 1 + \frac{1}{|V_{\mathcal{S}}|^2}\sum_{i=1}^n Z_{ji}'S_i^* \right|, \nonumber \\
		\quad b > 1, \forall j \in \mathcal{L}.
	\end{align}
	We immediately notice that increasing entries in the impedance matrix has exactly the same effects, so we have
	\begin{equation}
	|V_j'| < |V_j|, \qquad j \in \mathcal{L}.
	\end{equation} 
	Now that we have analyzed the impact of line impedance increase on the matrix $Z$ and bus voltage $V$, we can conclude that the $C$-index decreases with a homogeneous increase of line impedance as
	\begin{equation}
	C_j' = \frac{V_j'}{\sum_{i=1}^n \left| Z_{ji}'S_i/V_i' \right|} < C_j, \qquad \forall j \in \mathcal{L}.
	\end{equation}

	Next we consider the impact of load power factors on $C$-index. For simplicity, we again assume a homogeneous power factor variation such that the power factors of all buses $j \in \mathcal{L}$ decrease while the magnitudes of the apparent power remain constant. The change in load power factors affects the $C$-index through the change of load voltage profiles. Specifically, the angle $\angle \left( \sum_{i=1}^n Z_{ji}{S'}_i^* \right)$ for bus $j$ lies between $-180^{\circ}$ and $0^{\circ}$, and the decrease of power factors results in the decrease of $\angle \left( \sum_{i=1}^n Z_{ji}{S'}^* \right)$. Since the magnitudes of load power injections are constant, we have $\left| \sum_{i=1}^n Z_{ji}{S'}_i^* \right| = \left| \sum_{i=1}^n Z_{ji}S_i^* \right|$. As a result, the load voltage magnitude drops ($|V_j'| < |V_j|$) for all load bus $j$. Therefore, the $C$-indices of all load buses decrease. Note that the result here aligns with the general engineering wisdom that power factor correction can potentially benefit the system voltage stability.

	We demonstrate through two illustrative examples the impact of system parameters on the proposed $C$-index. Both examples show that the index provides a correct quantitative indication of the system stress level. The impact of other system parameters can be analyzed in a similar way. In particular, the analyses of the impact of changes of individual line impedance and load power factor can be performed as well, but are omitted for brevity.
	
	\subsection{Generalization to Systems with Constant Current Buses}
	\label{consti}
	DG inverters may be operated in either constant power or constant current modes. Therefore, the model should be able to represent generators as constant power and/or constant current buses. Since constant-current DGs are modeled as linear elements in the paper and their currents are given, their inclusion in the model does not change the dimension of the Jacobian matrix (\ref{comp_jacob}). 
	
	For example, consider the previous $(n + 1)$-bus power system model with $n$ PQ buses where bus $n+1$ is the slack bus. We augment the system by adding a constant current generator as bus $n+2$ and evaluate the change in (\ref{comp_jacob}). Recall the Jacobian (\ref{comp_jacob}) is
	\begin{equation}
	J^{\mathcal{Z}} = 
	\begin{bmatrix}
	\partial S / \partial I & \partial S / \partial I^* \\
	\partial S^* / \partial I & \partial S^* / \partial I^*
	\end{bmatrix}
	\end{equation}
	Note that the vectors $I = [I_1, \ldots, I_n]^T$ and $S = [S_1, \ldots, S_n]^T$ do not include the constant current bus $n+2$ and that the dimensions of the four submatrices $\partial S/ \partial I$, $\partial S/ \partial I^*$, $\partial S^*/ \partial I$, and $\partial S^*/ \partial I^*$ of $J^{\mathcal{Z}}$ are still $n \times n$. The expression of the entries of the matrix (\ref{dSdI}) remains the same. However, the impedance matrix $Z$ changes as a result of the modification of the system topology. For the diagonal matrix (\ref{dSdIstar}), the diagonal entries are modified by including the current injection from the constant current generator so that the $i$th diagonal element changes from
	\begin{equation}
	E_i - \sum\limits_{j=1}^nZ_{ij}I_j
	\end{equation}
	to
	\begin{equation}
	\left(E_i - Z_{i,n+2}I_{n+2}\right) - \sum\limits_{j=1}^nZ_{ij}I_j
	\end{equation}
	Hence, introducing constant current sources to the system can be considered as varying the equivalent source voltage seen from a PQ bus $i$ from $E_i$ to $E_i - Z_{i,n+2}I_{n+2}$ from the perspective of the complex power flow Jacobian in (\ref{comp_jacob}). Therefore, all the analysis made with the assumption of PQ buses apply.
	
	In this work we have considered constant-power and constant-current DGs, an important extension is to consider DGs with voltage regulation capability. The incorporation of these DGs modifies the equivalent voltage source $E$ seen by PQ buses in a similar way as what has been done to include constant-current DGs. Strictly speaking, if the DGs are set to regulate real power outputs and voltage magnitudes, then $E$ is no longer fixed as the phase angles of the DGs are free to vary. However, it turns out that the assumption of constant $E$ is a reasonable and effective approximation which has been extensively validated in voltage stability-related studies \cite{Liu1, Kessel1, WangY1} and can be used to incorporate voltage-regulated DGs in our framework.
	
	\section{Simulations} 
	\label{simulations}
	
	We perform case studies on a test distribution system that has been used in \cite{Bolognani2}.  Details of the test system can be found in \cite{Bolognani3}. The topology of the system is shown in Fig. \ref{fig1}. In the simulations, the DG penetration level is defined as the ratio of total DG capacity to total peak apparent load power of all loads \cite{Hoke1}. 
	
	\begin{figure}[!t]
		\centering
		\begin{tikzpicture}
		\draw[-][draw, thick] (0,0) -- (-0.6,0);
		\draw[-][draw, thick] (0,0) -- (0.6,0);
		\draw[-][draw, thick] (0.6,0) -- (1.2,0);
		\draw[-][draw, thick] (1.2,0) -- (1.8,0);
		\draw[-][draw, thick] (1.8,0) -- (2.4,0);
		\draw[-][draw, thick] (2.4,0) -- (3,0);
		\draw[-][draw, thick] (3,0) -- (3.6,0);
		\draw[-][draw, thick] (3.6,0) -- (4.2,0);
		\draw[-][draw, thick] (4.2,0) -- (4.8,0);
		
		\draw[-][draw, thick] (1.8,0) -- (1.8,0.6);
		\draw[-][draw, thick] (1.8,0.6) -- (1.8,1.2);
		\draw[-][draw, thick] (1.8,1.2) -- (1.8,1.8);
		\draw[-][draw, thick] (1.8,1.8) -- (1.8,2.4);
		\draw[-][draw, thick] (1.8,2.4) -- (1.8,3);
		\draw[-][draw, thick] (1.8,3) -- (1.8,3.6);
		\draw[-][draw, thick] (1.8,3.6) -- (1.8,4.2);
		
		\draw[-][draw, thick] (1.8,0.6) -- (2.4,0.6);
		\draw[-][draw, thick] (2.4,0.6) -- (2.4,1.2);
		\draw[-][draw, thick] (2.4,1.2) -- (2.4,1.8);
		\draw[-][draw, thick] (2.4,1.8) -- (2.4,2.4);
		\draw[-][draw, thick] (2.4,2.4) -- (2.4,3.0);
		\draw[-][draw, thick] (2.4,3.0) -- (2.4,3.6);
		\draw[-][draw, thick] (2.4,3.0) -- (3.0,3.0);
		\draw[-][draw, thick] (3.0,3.0) -- (3.6,3.0);
		\draw[-][draw, thick] (3.6,3.0) -- (4.2,3.0);
		
		\draw[-][draw, thick] (3.6,0.0) -- (3.6,0.6);
		\draw[-][draw, thick] (3.6,0.6) -- (4.8,0.6);
		\draw[-][draw, thick] (4.8,0.6) -- (4.8,1.2);
		\draw[-][draw, thick] (4.8,1.2) -- (4.8,1.8);
		\draw[-][draw, thick] (4.8,1.8) -- (4.8,2.4);
		\draw[-][draw, thick] (4.8,2.4) -- (4.2,2.4);
		\draw[-][draw, thick] (4.2,2.4) -- (3.6,2.4);
		
		\draw[-][draw, thick] (4.8,0.6) -- (5.4,0.6);
		\draw[-][draw, thick] (5.4,0.6) -- (5.4,1.2);
		\draw[-][draw, thick] (5.4,1.2) -- (5.4,1.8);
		\draw[-][draw, thick] (5.4,1.8) -- (5.4,2.4);
		\draw[-][draw, thick] (5.4,2.4) -- (5.4,3.0);
		\draw[-][draw, thick] (5.4,1.2) -- (6.0,1.2);
		\draw[-][draw, thick] (6.0,1.2) -- (6.6,1.2);
		\draw[-][draw, thick] (6.6,1.2) -- (7.2,1.2);
		
		\draw[-][draw, thick] (5.4,0.6) --  (5.4,0.0);
		\draw[-][draw, thick] (5.4,0.0) --  (5.4,-0.6);
		\draw[-][draw, thick] (5.4,-0.6) -- (4.8,-0.6);
		\draw[-][draw, thick] (4.8,-0.6) -- (4.2,-0.6);
		\draw[-][draw, thick] (4.2,-0.6) -- (3.6,-0.6);
		\draw[-][draw, thick] (3.6,-0.6) -- (3.0,-0.6);
		\draw[-][draw, thick] (3.0,-0.6) -- (2.4,-0.6);
		\draw[-][draw, thick] (2.4,-0.6) -- (1.8,-0.6);
		\draw[-][draw, thick] (5.4,-0.6) -- (6.0,-0.6);
		\draw[-][draw, thick] (6.0,-0.6) -- (6.6,-0.6);
		\draw[-][draw, thick] (6.6,-0.6) -- (7.2,-0.6);
		\draw[-][draw, thick] (6.6,-0.6) -- (6.6,0.0);
		\draw[-][draw, thick] (6.6,0.0) --  (6.6,0.6);
		\draw[-][draw, thick] (6.6,0.6) --  (7.2,0.6);
		\draw[-][draw, thick] (7.2,0.6) --  (7.8,0.6);

		\draw[draw, very thick] (-0.6,-.3) -- (-0.6, .3);
		\node [below] at (-0.6,-0.3) {\small 56};
		\draw [fill] (0,0) circle [radius=0.05];
		\node [below] at (0,0) {\small 1};
		\draw [fill] (0.6,0) circle [radius=0.05];
		\node [below] at (0.6,0) {\small 2};
		\draw [fill] (1.2,0) circle [radius=0.05];
		\node [below] at (1.2,0) {\small 3};
		\draw [fill] (1.8,0) circle [radius=0.05];
		\node [below] at (1.8,0) {\small 4};
		\draw [fill] (2.4,0) circle [radius=0.05];
		\node [below] at (2.4,0) {\small 5};
		\draw [fill] (3,0) circle [radius=0.05];
		\node [below] at (3.0,0) {\small 6};
		\draw [fill] (3.6,0) circle [radius=0.05];
		\node [below] at (3.6,0) {\small 7};
		\draw [fill] (4.2,0) circle [radius=0.05];
		\node [below] at (4.2,0) {\small 8};
		\draw [fill=red,red] (4.73,-0.07) rectangle (4.87,0.07);
		\node [below] at (4.8,0) {\small 9};
		
		\draw [fill] (1.8,0.6) circle [radius=0.05];
		\node [left] at (1.8,0.6) {\small 40};
		\draw [fill] (1.8,1.2) circle [radius=0.05];
		\node [left] at (1.8,1.2) {\small 41};
		\draw [fill] (1.8,1.8) circle [radius=0.05];
		\node [left] at (1.8,1.8) {\small 42};
		\draw [fill=red, red] (1.73,2.33) rectangle (1.87,2.47);
		\node [left] at (1.8,2.4) {\small 43};
		\draw [fill] (1.8,3.0) circle [radius=0.05];
		\node [left] at (1.8,3.0) {\small 44};
		\draw [fill] (1.8,3.6) circle [radius=0.05];
		\node [left] at (1.8,3.6) {\small 45};
		\draw [fill] (1.8,4.2) circle [radius=0.05];
		\node [left] at (1.8,4.2) {\small 46};
		
		\draw [fill] (2.4,0.6) circle [radius=0.05];
		\node [right] at (2.4,0.6) {\small 47};
		\draw [fill=red,red] (2.33,1.13) rectangle (2.47, 1.27);
		\node [right] at (2.4,1.2) {\small 48};
		\draw [fill] (2.4,1.8) circle [radius=0.05];
		\node [right] at (2.4,1.8) {\small 49};
		\draw [fill] (2.4,2.4) circle [radius=0.05];
		\node [right] at (2.4,2.4) {\small 50};
		\draw [fill=red,red] (2.33,2.93) rectangle (2.47, 3.07);
		\node [below right] at (2.4,3) {\small 51};
		\draw [fill] (2.4,3.6) circle [radius=0.05];
		\node [above] at (2.4,3.6) {\small 52};
		\draw [fill] (3,3) circle [radius=0.05];
		\node [above] at (3.0,3.0) {\small 53};
		\draw [fill] (3.6,3.0) circle [radius=0.05];
		\node [above] at (3.6,3.0) {\small 54};
		\draw [fill] (4.2,3.0) circle [radius=0.05];
		\node [above] at (4.2,3.0) {\small 55};
		
		\draw [fill] (3.6,0.6) circle [radius=0.05];
		\node [above] at (3.6,0.6) {\small 10};
		\draw [fill] (4.8,0.6) circle [radius=0.05];
		\node [above left] at (4.8,0.6) {\small 11};
		\draw [fill] (4.8,1.2) circle [radius=0.05];
		\node [left] at (4.8,1.2) {\small 12};
		\draw [fill] (4.8,1.8) circle [radius=0.05];
		\node [left] at (4.8,1.8) {\small 13};
		\draw [fill] (4.8,2.4) circle [radius=0.05];
		\node [above] at (4.8,2.4) {\small 14};
		\draw [fill=red,red] (4.13,2.33) rectangle (4.27,2.47);
		\node [above] at (4.2,2.4) {\small 15};
		\draw [fill] (3.6,2.4) circle [radius=0.05];
		\node [above] at (3.6,2.4) {\small 16};
		
		\draw [fill] (5.4,0.6) circle [radius=0.05];
		\node [right] at (5.4,0.6) {\small 17};
		\draw [fill] (5.4,1.2) circle [radius=0.05];
		\node [left] at (5.4,1.2) {\small 33};
		\draw [fill] (5.4,1.8) circle [radius=0.05];
		\node [right] at (5.4,1.8) {\small 37};
		\draw [fill=red,red] (5.33,2.33) rectangle (5.47,2.47);
		\node [right] at (5.4,2.4) {\small 38};
		\draw [fill] (5.4,3.0) circle [radius=0.05];
		\node [right] at (5.4,3.0) {\small 39};
		\draw [fill] (6.0,1.2) circle [radius=0.05];
		\node [above] at (6.0,1.2) {\small 34};
		\draw [fill=red,red] (6.53,1.13) rectangle (6.67,1.27);
		\node [above] at (6.6,1.2) {\small 35};
		\draw [fill] (7.2,1.2) circle [radius=0.05];
		\node [above] at (7.2,1.2) {\small 36};
		
		\draw [fill=red,red] (5.33,-0.07) rectangle (5.47,0.07);
		\node [right] at (5.4,0.0) {\small 18};
		\draw [fill] (5.4,-0.6) circle [radius=0.05];
		\node [below] at (5.4,-0.6) {\small 19};
		\draw [fill] (4.8,-0.6) circle [radius=0.05];
		\node [below] at (4.8,-0.6) {\small 27};
		\draw [fill] (4.2,-0.6) circle [radius=0.05];
		\node [below] at (4.2,-0.6) {\small 28};
		\draw [fill] (3.6,-0.6) circle [radius=0.05];
		\node [below] at (3.6,-0.6) {\small 29};
		\draw [fill] (3.0,-0.6) circle [radius=0.05];
		\node [below] at (3.0,-0.6) {\small 30};
		\draw [fill] (2.4,-0.6) circle [radius=0.05];
		\node [below] at (2.4,-0.6) {\small 31};
		\draw [fill=red,red] (1.73,-0.67) rectangle (1.87, -0.53);
		\node [below] at (1.8,-0.6) {\small 32};
		\draw [fill] (6.0,-0.6) circle [radius=0.05];
		\node [below] at (6.0,-0.6) {\small 20};
		\draw [fill] (6.6,-0.6) circle [radius=0.05];
		\node [below] at (6.6,-0.6) {\small 21};
		\draw [fill] (7.2,-0.6) circle [radius=0.05];
		\node [below] at (7.2,-0.6) {\small 22};
		\draw [fill] (6.6,0.0) circle [radius=0.05];
		\node [left] at (6.6,0.0) {\small 23};
		\draw [fill=red,red] (6.53,0.53) rectangle (6.67,0.67);
		\node [left] at (6.6,0.6) {\small 24};
		\draw [fill] (7.2,0.6) circle [radius=0.05];
		\node [below] at (7.2,0.6) {\small 25};
		\draw [fill] (7.8,0.6) circle [radius=0.05];
		\node [below] at (7.8,0.6) {\small 26};
		
		\begin{customlegend}[
		legend entries={ % <= in the following there are the entries
			Slack bus,
			Generator bus,
			Load bus
		},
		legend style={at={(8,5)},font=\footnotesize}] % <= to define position and font legend
		% the following are the "images" and numbers in the legend
		\addlegendimage{fill=black!100!red,sharp plot}
		\addlegendimage{only marks, mark=square*,fill=black!0!red, draw=white,sharp plot}
		\addlegendimage{mark=*,ball color=black,draw=white,sharp plot}
		\end{customlegend}

		\end{tikzpicture}
		\caption{Modified 123-Bus System from \cite{Bolognani2}.} 
		\label{fig1}
	\end{figure}
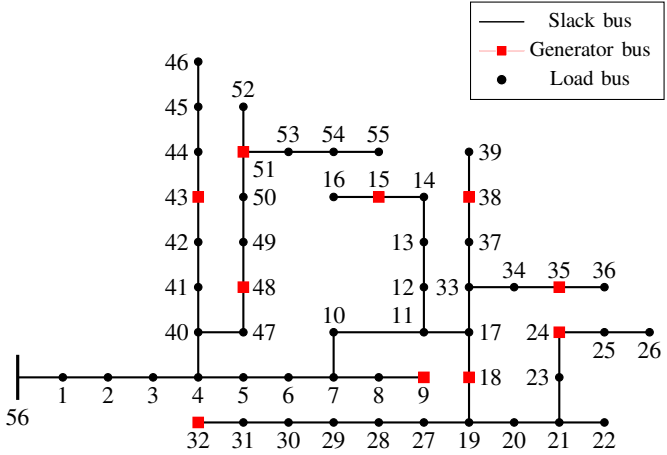
	
	\subsection{System With Constant-Power DG inverters}
	
	Simulations are performed by increasing the load power consumption at each load bus incrementally at the step of 1\% of the base load until power flow fails to converge. Power flow analysis is performed using the open-source package Matpower \cite{Zimmerman1}. The base load of the given system can be found in \cite{Bolognani3}. Table \ref{table1} compares the actual critical loading factors obtained through power flow analysis and those calculated by the proposed $C$-index at different DG penetration levels. 
	\begin{table}[!t]
		\renewcommand{\arraystretch}{1.3}
		\caption{Comparison of Power Flow and $C$-index-based Critical Loading Factors}
		\centering
		\begin{tabular}{cccc}
			\hline
			\multirow{2}{*}{DG penetration} & \multirow{2}{*}{Critical} & Loading factor  & Difference of  \\
			& & when $C$-index  & two loading  \\
			level & loading factor & drops to 1 & factors \\
			\hline
			10\% 	& 4.251 & 4.189 & 1.46\%	\\
			20\% 	& 4.332 & 4.270 & 1.43\%	\\
			30\% 	& 4.413 & 4.348 & 1.47\%	\\
			40\% 	& 4.494 & 4.424 & 1.56\%	\\
			50\% 	& 4.574 & 4.497 & 1.68\%	\\
			60\% 	& 4.654 & 4.566 & 1.89\%	\\
			70\% 	& 4.733 & 4.632 & 2.13\%	\\
			80\% 	& 4.811 & 4.695 & 2.41\%	\\
			90\% 	& 4.889 & 4.755 & 2.74\%	\\
			100\% 	& 4.967 & 4.811 & 3.14\%	\\
			\hline
		\end{tabular}
		\label{table1}
	\end{table}	
%	Results from applying the necessary condition are shown in Table \ref{table1}. A wide range of penetration levels are applied and are shown in the first column of the table. The penetration level is defined as the ratio of total DG power to load apparent power on the feeder \cite{Hoke1}, that is,
%	\begin{equation}
%	\text{DG Penetration level} = \frac{\text{Peak DG power generation}}{\text{Peak load apparent power}}
%	\end{equation}
	The second column shows the loading factor at which the power flow diverges, i.e., actual critical loading factors. The third column shows the loading factors at which the system-wide $C$-index (i.e., the minimum $C$-index) reaches 1. It is observed that for all cases, the point of the first occurrence of unity $C$-index lies close to the power flow solvability boundary. 
	
%	\begin{figure}[!t]
%		\centering
%		\includegraphics[width=3.5in]{J_bar1.eps}
%		% where an .eps filename suffix will be assumed under latex, 
%		% and a .pdf suffix will be assumed for pdflatex; or what has been declared
%		% via \DeclareGraphicsExtensions.
%		\caption{Visualization of complex Jacobian matrix at base case.}
%		\label{fig2}
%	\end{figure}
%	
%	\begin{figure}[!t]
%		\centering
%		\includegraphics[width=3.5in]{J_bar2.eps}
%		% where an .eps filename suffix will be assumed under latex, 
%		% and a .pdf suffix will be assumed for pdflatex; or what has been declared
%		% via \DeclareGraphicsExtensions.
%		\caption{Visualization of complex Jacobian matrix at critical point.}
%		\label{fig3}
%	\end{figure}

%	\begin{figure}[!t]
%		\centering
%		\subfloat[Base case]{{\includegraphics[width=3in]{J_bar1} }\label{fig2a}}%
%		\qquad
%		\subfloat[Critical point]{{\includegraphics[width=3in]{J_bar2} }\label{fig2b}}%
%		\caption{Visualization of matrix $F$ at (a) base case and (b) critical point with penetration level of 90\%}%
%		\label{fig2}%
%	\end{figure}
	
	\begin{figure}[!t]
		\centering
		\begin{tikzpicture}
		\node[inner sep=0pt] (russell) at (0,0)
		{\includegraphics[width=3in]{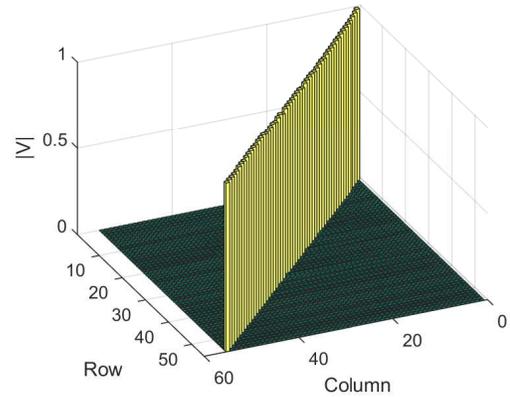}};
		\node [below] at (0,-2.8) {\small (a) Base case};
		\node[inner sep=0pt] (whitehead) at (0,-6.4)
		{\includegraphics[width=3in]{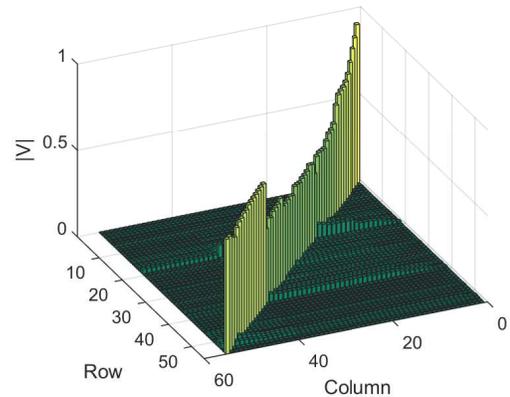}};
		\node [below] at (0,-9.2) {\small (b) Critical point};
		\end{tikzpicture}
		\caption{Visualization of matrix $F$ at (a) base case and (b) critical point with penetration level of 90\%}
		\label{fig2}
	\end{figure}
	
	Fig. \ref{fig2} shows the matrix $J^\mathcal{Z'}$ defined in (\ref{Jzp}) when the DG penetration level is 90\%. $J^\mathcal{Z'}$ is shown since it is proved in Lemma \ref{lemma2} that the determinant of $J^\mathcal{Z'}$ has the same magnitude as the determinant of the complex power flow Jacobian $J^\mathcal{Z}$ in (\ref{comp_jacob}). Instead of showing the full $2n \times 2n$ matrix $J^\mathcal{Z'}$ whose lower blocks are complex conjugate to their upper counterparts, we show a more compact real-valued $n \times n$ matrix $F$ defined as
	\begin{equation}
	F = \left| \left|\frac{\partial S}{\partial I^*}\right| - |B| \right|
	\end{equation}
	where $\partial S/ \partial I^*$ and $B$ are two upper submatrices of $J^\mathcal{Z'}$ and $|\cdot|$ denotes a matrix whose entries are element-wise magnitudes of the original matrix. Notice that the off-diagonal entries of $F$ are the magnitudes of the corresponding entries of the matrix $B$, $F_{ij} = |Z_{ij}I_{j}|, i \neq j$, whereas the diagonal entries of $F$ are the differences between the magnitudes of the diagonal entries of $\partial S/ \partial I^*$ and that of $B$, $F_{ii} = |V_i| - |Z_{ii}I_i|$. That is, $F$ is diagonally dominant if and only if $J^\mathcal{Z'}$ is diagonally dominant, given that the magnitude of the diagonal entries of $\partial S/ \partial I^*$ is larger than that of $B$. 
	
	It is observed from Fig. \ref{fig2}a that, at the base loading condition, $F$ is strongly diagonally dominant in the sense that the diagonal entries are much larger than the sum of the off-diagonal entries. The matrix $F$ gradually loses its diagonal dominance as system stress increases with the decrease of diagonal elements (bus voltage) and increase of off-diagonal elements (voltage drop). This can be seen in Fig. \ref{fig2}b, which shows the matrix $F$ at the critical loading condition.
	
	\begin{figure}[!t]
	\centering
	\includegraphics[width=3.5in]{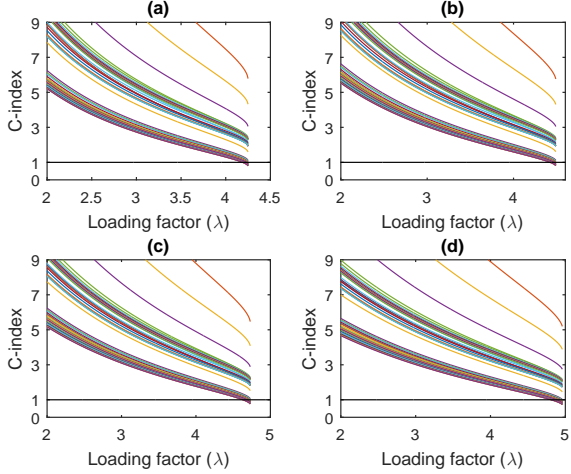}
		% where an .eps filename suffix will be assumed under latex, 
		% and a .pdf suffix will be assumed for pdflatex; or what has been declared
		% via \DeclareGraphicsExtensions.
	\caption{Change of $C$-index at all load buses as system stress increases when system penetration levels are (a) 10\%, (b) 40\%, (c) 70\%, and (d) 100\%.}
	\label{fig3}
	\end{figure}
	
	\begin{figure}[!t]
		\centering
		\begin{tikzpicture}
		\node[inner sep=0pt] (russell) at (0,0)
		{\includegraphics[width=3.7in]{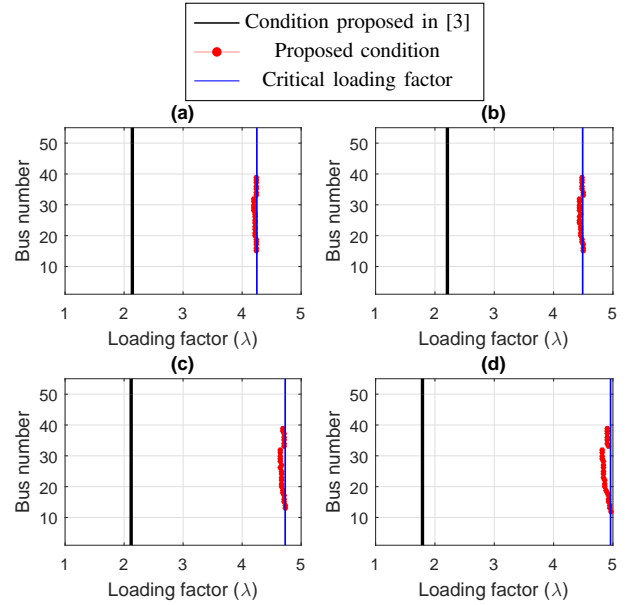}};
		
		\begin{customlegend}[
		legend entries={ % <= in the following there are the entries
			Condition proposed in \cite{Bolognani2},
			Proposed condition,
			Critical loading factor
		},
		legend style={at={(2.1,4.6)},font=\footnotesize}] % <= to define position and font legend
		% the following are the "images" and numbers in the legend
		\addlegendimage{fill=black!100!red,sharp plot}
		\addlegendimage{only marks, mark=*,fill=black!0!red, draw=white,sharp plot},
		\addlegendimage{no markers,blue}
		\end{customlegend}
		\end{tikzpicture}
		% where an .eps filename suffix will be assumed under latex, 
		% and a .pdf suffix will be assumed for pdflatex; or what has been declared
		% via \DeclareGraphicsExtensions.
		\caption{The red dots show where $C$-index drops to unity for each bus before critical point when system penetration levels are (a) 10\%, (b) 40\%, (c) 70\%, and (d) 100\%. Black lines show the loading factor when the sufficient condition of solution existence proposed in \cite{Bolognani2} is marginally satisfied. Blue lines show the critical loading factor.}
		\label{fig4}
	\end{figure}
	
	Fig. \ref{fig3} shows the changes of $C$-indices at all PQ buses as the system loading factor increases when DG penetration levels are 10\%, 40\%, 70\% and 100\%, respectively. It can be seen that the indices are monotonically decreasing as system stresses and more than one-third of them drop below unity at the critical point.
	
	Fig. \ref{fig4} shows, by red dots, the point where $C$-index drops to 1 for each bus when system penetration levels are 10\%, 40\%, 70\%, and 100\%, respectively. The buses whose corresponding $C$-indices drop below 1 are in the range of bus numbers 10--40, indicating the severity of stress for those buses. The closeness of the first occurrence of unity index to the actual critical point is demonstrated. The vertical solid lines show the loading factor when the sufficient condition of solution existence proposed in \cite{Bolognani2} is marginally satisfied. It can be seen that the power flow is still solvable when the loading factor exceeds the vertical solid line. However, it becomes insolvable around the red dots. Therefore, the condition given in \cite{Bolognani2} is more conservative, which is in accordance with the argument in Remark \ref{rmk2}.
	
%	To validate that the complex power flow Jacobian matrix $J^{\mathcal{Z}}$ defined in (\ref{comp_jacob}) indeed becomes singular at the loadability limit, Figure \ref{fig1} shows the variation of the smallest eigenvalue of $J^{\mathcal{Z}}$ as system loading factor increases for penetration levels of 10\%, 30\%, 50\%, 70\%, and 90\%. It is seen from the figure that the eigenvalues approach 0 as the system approaches critical point.

	\subsection{System With Constant-Current DG Inverters}
	
	To demonstrate the effectiveness of the proposed method in analyzing a system with constant-current DG inverters, we replace the constant-power sources at buses 9, 24, 35, 43, and 51 with constant-current ones. Fig. \ref{fig5} shows $C$-indices at all buses at the system critical point when DG penetration levels are (a) 10\%, (b) 40\%, (c) 70\%, and (d) 100\%. There are vacancies in the figures because buses with constant current sources are removed from the complex power flow Jacobian $J^{\mathcal{Z}}$, and their $C$-indices are not calculated. It can be seen that some indices are less than 1 at the critical loading condition, thus validating the extension of the proposed condition to systems with constant current buses.
	\begin{figure}[!t]
		\centering
		\includegraphics[width=3.5in]{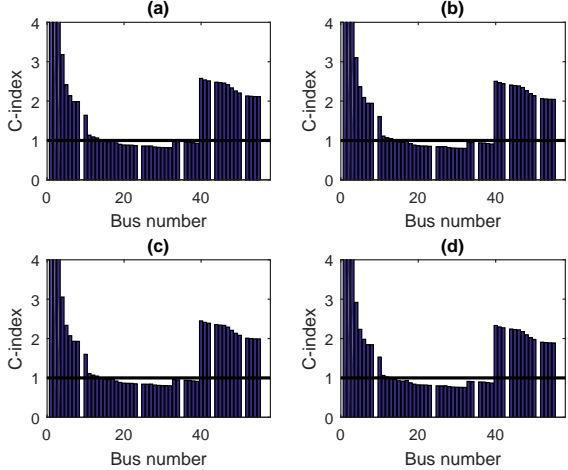}
		% where an .eps filename suffix will be assumed under latex, 
		% and a .pdf suffix will be assumed for pdflatex; or what has been declared
		% via \DeclareGraphicsExtensions.
		\caption{$C$-index at all PQ buses at critical point when DG penetration level are (a) 10\%, (b) 40\%, (c) 70\%, and (d) 100\%. Buses 9, 24, 35, 43, and 51 are constant current sources.}
		\label{fig5}
	\end{figure}
	
	\begin{table}[!t]
		\renewcommand{\arraystretch}{1.3}
		\caption{Comparison of Power Flow and $C$-index-based Critical Loading Factors for Systems with Constant-Current DGs}
		\centering
		\begin{tabular}{cccc}
			\hline
			\multirow{2}{*}{DG penetration} & \multirow{2}{*}{Critical} & Loading factor  & Difference of  \\
			& & when $C$-index  & two loading  \\
			level & loading factor & drops to 1 & factors \\
			\hline
			10\% 	& 4.522 & 4.456 & 1.46\%	\\
			20\% 	& 4.589 & 4.523 & 1.44\%	\\
			30\% 	& 4.656 & 4.588 & 1.46\%	\\
			40\% 	& 4.721 & 4.651 & 1.48\%	\\
			50\%	& 4.786 & 4.712 & 1.55\%	\\
			60\% 	& 4.851 & 4.771 & 1.65\%	\\
			70\% 	& 4.915 & 4.828 & 1.77\%	\\
			80\% 	& 4.978 & 4.883 & 1.91\%	\\
			90\% 	& 5.041 & 4.937 & 2.06\%	\\
			100\% 	& 5.103 & 4.989 & 2.23\%	\\
			\hline
		\end{tabular}
		\label{table2}
	\end{table}
	
	Table \ref{table2} compares the actual critical loading factors obtained through power flow analysis and those calculated by the proposed $C$-index at different DG penetration levels for a system with constant-current DGs. It is seen from the table that the differences of the two loading factors induced by systems with constant-current DGs are smaller compared to systems with only constant-power DGs.
	
	\section{Discussions}
	\label{discussion}
	In this section, we discuss the reasons of the minor mismatch of the proposed and actual critical loading factors, compare the stability indices for systems with different types of DGs, and introduce the physical meaning of the proposed condition.

	First of all, we point out that the critical loading factor provided by the proposed condition is closer to the actual one when loads are changing proportionally. This can be explained as follows: consider the system introduced in Section \ref{modelling} with $n+1$ buses where bus 0 is the slack bus and buses 1 to $n$ are PQ buses. If the power injections of the PQ buses change in a way such that their current injections are proportional (magnitude- and angle-wise), then Kessel and Glavitsch \cite{Kessel1} showed that the steady-state voltage instability occurs when there is a bus $j$ such that 
	\begin{equation}
	\left| \sum_{i=1}^n Z_{ji}I_i \right| = |V_j|.
	\end{equation}
	However, the above condition only holds when PQ bus current injections are always proportional, which is unrealistic. In particular, the condition may be met either before or after actual voltage stability point if the PQ bus current injections are disproportional, which is a major drawback. However, when power injections are proportional, the assumption of proportional current injection is approximately satisfied since bus voltages are close to 1 under normal operating conditions and their changes tend to be homogeneous as well. Therefore, the condition in \cite{Kessel1} works relatively well under proportional load variations, and it becomes less effective as load variation deviates from the assumed proportional pattern.

	Notice the similarity between the condition in \cite{Kessel1} and our proposed sufficient condition for voltage stability,
	\begin{equation}
	\sum_{i=1}^n \left| Z_{ji}I_i \right| < |V_j|, \qquad j = 1, \ldots, n.
	\label{cond_repeat}
	\end{equation}
	In fact, the proposed condition can be regarded as a generalization of the one in \cite{Kessel1}. By the same token, the mismatch between the critical loading factors given by the proposed condition and the actual one is smaller when the power injections are proportional, and becomes larger when the disproportionality of power injections increases. As a special case, the critical loading factor provided by the proposed condition coincides with that in \cite{Kessel1} when all lines in the system have the same $r/x$ ratio and all PQ bus current injections have an identical phase angle. This is because under these assumptions all summands $Z_{ji}I_i$ on the left side in the proposed condition (\ref{cond_repeat}) are in phase and the absolute value operator can be moved outside the summation. In this manuscript, constant-power DGs in the system are simulated such that their outputs remain unchanged as load demands increase. As such, the power injection disproportionality rises with an increasing penetration level of constant-power DGs. So the difference between the two loading factors becomes larger with the increase of PV penetration level as shown in Table \ref{table1}.
	
	On the other hand, the penetrations of constant-current DGs do not affect the proportionality of power injections, since constant-current DGs are linear elements from circuit analysis perspective and are not included in constant-power buses. Hence, their current contributions are not included in the left side of the proposed condition (\ref{cond_repeat}). Rather, the contributions of constant-current DGs are modeled as a modification to the equivalent voltage source $E$ as explained in Section \ref{consti}. Therefore, the mismatch between the proposed and actual loading factors is larger for constant-power DGs since their penetration leads to disproportional variations of power injections.
	
	However, the difference between the two loading factors does not necessarily reflect the accuracy of the proposed condition. The system critical loading factors in Tables \ref{table1} and \ref{table2} are obtained by assuming a specific power variation pattern (i.e., constant DG power/current injection and proportionally increasing load demands in this paper). For instance, if outputs of constant-power DGs decrease as load demands increase, then the critical loading factor would be smaller. Nevertheless, the proposed index provides a lower bound for the \emph{smallest} critical loading factor. The comparison of the critical loading factor by the proposed method and the worst critical loading factor is beyond the scope of the manuscript. However, our main point is that by relaxing the condition proposed in \cite{Kessel1}, we have proved rigorously that the new condition (\ref{cond_repeat}) guarantees voltage stability under all power variation patterns, not only when the constant-power bus current injections are changing proportionally.
	
%	It is well-known that PV-PQ transitions of PV generators has adverse effect on system loadability, generators encountering reactive power limit may even lead to sudden voltage collapse in the form of limit-induced bifurcation \cite{Dobson1}. By treating PV generators as constant power buses, the proposed index can also be used to evaluate the loadability of systems containing PV generators, though the results might be slightly conservative. 
	
	\section{Conclusions}
	\label{conclsn}
	This paper proposes a necessary condition for the power flow insolvability in a distribution system with DGs. The condition is proved through detailed mathematical derivation. It is shown that the proposed condition provides a precursor for power flow insolvability by setting an operating condition-dependent upper bound in the power injection space. Based on the necessary condition, a new index is designed to monitor the operating condition and it provides a precursor to voltage instability. We verify the effectiveness of the proposed condition and index via numerical simulations on a distribution test system with different types and penetration levels of DGs. The advantages of the proposed method can be summarized as 1) it is adaptive to system operating conditions, 2) the calculation only needs the present snapshot of voltage phasors, and 3) it requires a small computation effort. The proposed method can be used to assist the planning, online monitoring and operation of power distribution systems with DGs.

	\ifCLASSOPTIONcaptionsoff
	\newpage
	\fi

	
	
	%\bibliographystyle{IEEEtran}

	\begin{IEEEbiography}[{\includegraphics[width=1in,height=1.25in,clip,keepaspectratio]{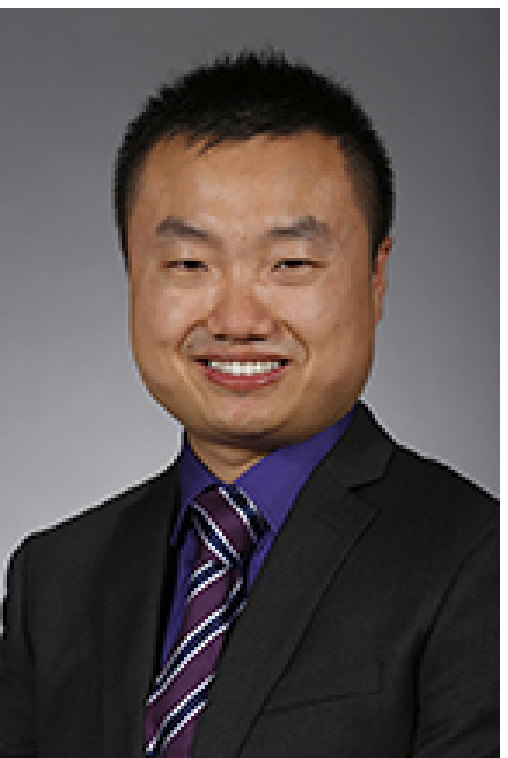}}]{Zhaoyu Wang}
		is an Assistant Professor at Iowa State Unviersity. He received Ph.D. degree in Electrical and Computer Engineering from Georgia Institute of Technology in 2015. He received B.S. degree and M.S. degree in Electrical Engineering from Shanghai Jiao Tong University in 2009 and 2012, respectively, and the M.S. degree in Electrical and Computer Engineering from Georgia Institute of Technology in 2012. His research interests include power distribution systems, microgrids, renewable integration, self-healing resilient power systems, and voltage/VAR control. He was a Research Aid in 2013 at Argonne National Laboratory and an Electrical Engineer at Corning Incorporated in 2014.
	\end{IEEEbiography}	
	
	\begin{IEEEbiography}[{\includegraphics[width=1in,height=1.25in,clip,keepaspectratio]{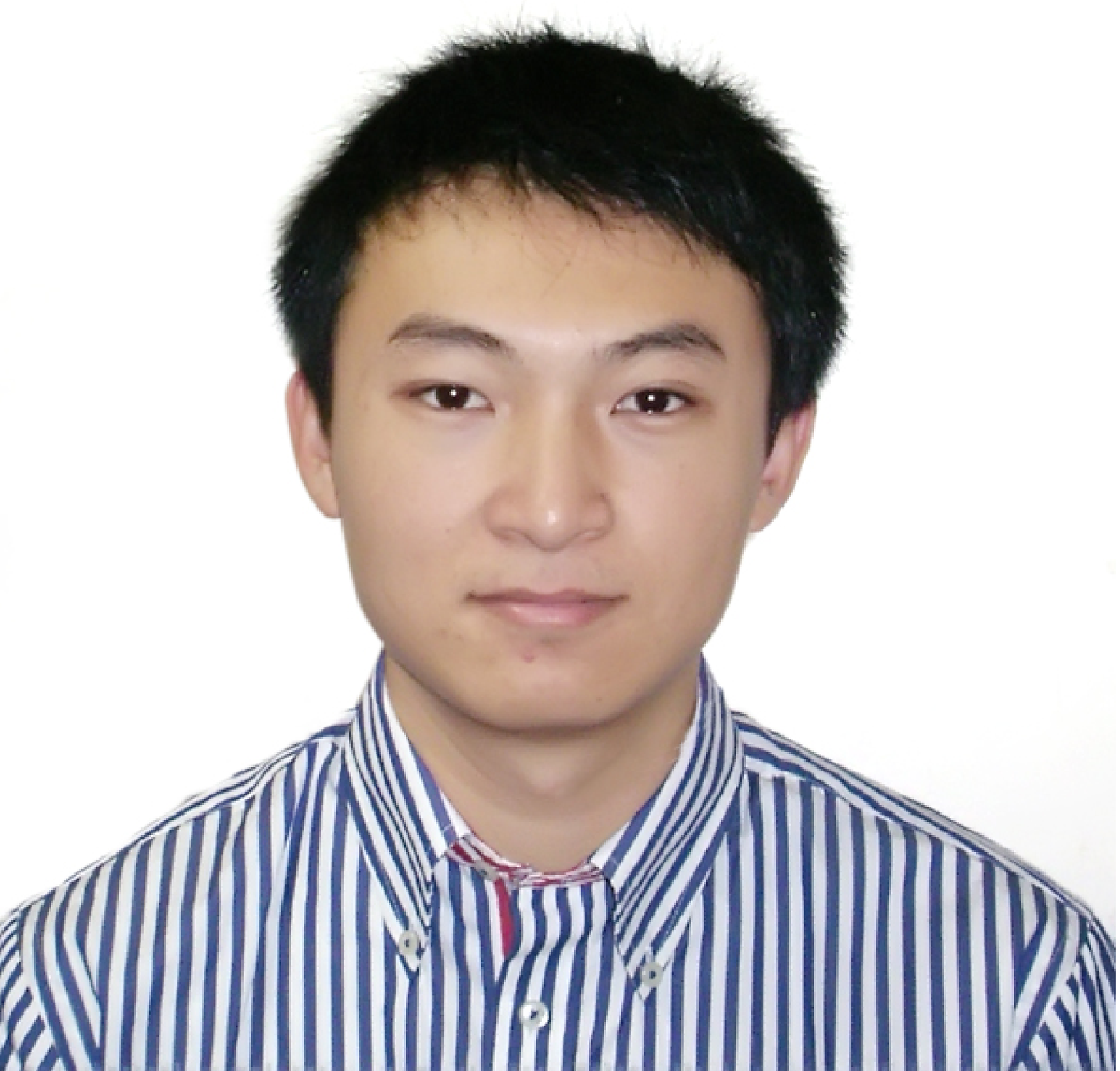}}]{Bai Cui}
		received the B.S. degree in Electrical Engineering from Shanghai Jiao Tong University, Shanghai, China in 2011, and the B.S. degree in Computer Engineering from the University of Michigan, Ann Arbor, MI, in 2011, and the M.S. degree in Electrical and Computer Engineering from the Georgia Institute of Technology, Atlanta, GA, in 2014. He is currently working towards the Ph.D. degree in the School of Electrical and Computer Engineering, Georgia Institute of Technology. His research interests include renewable integration and power system stability and control.
	\end{IEEEbiography}
	
	\begin{IEEEbiography}[{\includegraphics[width=1in,height=1.25in,clip,keepaspectratio]{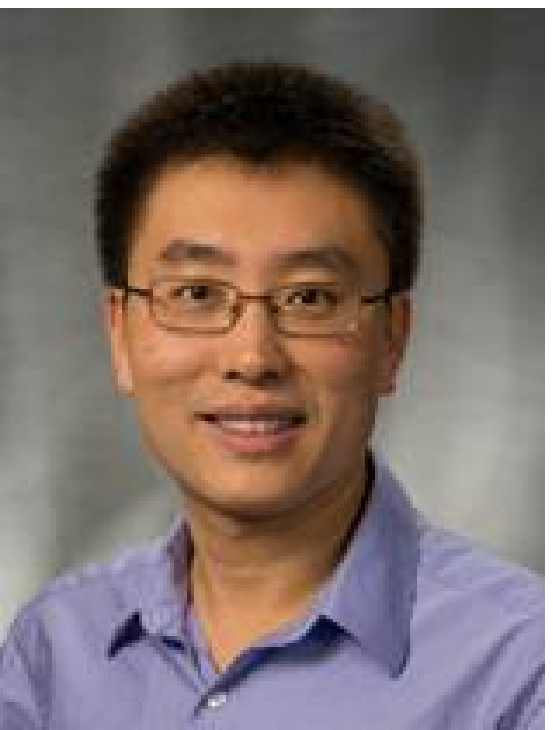}}]{Jianhui Wang}
		(M'07--SM'12) received the Ph.D. degree in electrical engineering from Illinois Institute of Technology, Chicago, IL, USA, in 2007.
		Presently, he is the Section Lead for Advanced Power Grid Modeling at the Energy Systems Division at Argonne National Laboratory, Argonne, IL, USA.
		
		Dr. Wang is the secretary of the IEEE Power \& Energy Society (PES) Power System Operations Committee. He is an associate editor of Journal of Energy Engineering and an editorial board member of Applied Energy. He is also an affiliate professor at Auburn University and an adjunct professor at University of Notre Dame. He has held visiting positions in Europe, Australia and Hong Kong including a VELUX Visiting Professorship at the Technical University of Denmark (DTU). Dr. Wang is the Editor-in-Chief of the \textsc{IEEE Transactions on Smart Grid} and an IEEE PES Distinguished Lecturer. He is also the recipient of the IEEE PES Power System Operation Committee Prize Paper Award in 2015.
	\end{IEEEbiography}

	
	%\end{IEEEbiography}
	
	
	%\newpage
	
	
	
	%\vfill
	

\end{document}